\documentclass[10pt]{amsart}
\usepackage{amsmath}
\usepackage{amsxtra}
\usepackage{amscd}
\usepackage{amsthm}
\usepackage{amsfonts}
\usepackage{amssymb}
\usepackage{eucal}
\usepackage{mathabx}
\usepackage[matrix,arrow,curve]{xy}
\usepackage{pb-diagram}
\usepackage{comment}
\usepackage{array}
\newcommand{\fg}{{\mathfrak{g}}}
\newcommand{\bp}{{\mathbb{P}}}
\newcommand{\fh}{{\mathfrak{h}}}
\newcommand{\fl}{{\mathfrak{l}}}
\newcommand{\fp}{{\mathfrak{p}}}
\newcommand{\fz}{{\mathfrak{z}}}
\newcommand{\fsl}{{\mathfrak{sl}}}
\newcommand{\fgl}{{\mathfrak{gl}}}
\newcommand{\tr}{{\rm tr} \,}
\newcommand{\pr}{{\rm pr} \,}
\newcommand{\Tr}{{\rm Tr} \,}

\newcommand{\End}{{\rm End} \,}

\newcommand{\Ad}{{\rm Ad} \,}
\newcommand{\rk}{{\rm rk} \,}
\newcommand{\gr}{{\rm gr} \,}

\newcommand{\bc}{\mathbb{C}}
\newcommand{\bz}{\mathbb{Z}}

\newcommand{\oG}{\overline{G}}
\newcommand{\diag}{{\rm diag} \,}

\newtheorem{defn}[subsection]{Definition}
\newtheorem{thm}[subsection]{Theorem}
\newtheorem{lem}[subsection]{Lemma}
\newtheorem{prop}[subsection]{Proposition}

\newtheorem{cor}[subsection]{Corollary}
\newtheorem{conj}[subsection]{Conjecture}

\newtheorem*{rem}{Remark}
\newtheorem{remm}[subsection]{Remark}

\title{Bethe subalgebras in Yangians and the wonderful compactification}
\author{Aleksei Ilin and Leonid Rybnikov}

\begin{document} 
\begin{abstract}
Let $\fg$ be a complex simple Lie algebra. We study the family of Bethe subalgebras in the Yangian $Y(\fg)$ parameterized by the corresponding adjoint Lie group $G$. We describe their classical limits as subalgebras in the algebra of polynomial functions on the formal Lie group $G_1[[t^{-1}]]$. In particular we show that, for regular values of the parameter, these subalgebras are free polynomial algebras with the same Poincar\'e series as the Cartan subalgebra of the Yangian. Next, we extend the family of Bethe subalgebras to the De Concini--Procesi wonderful compactification $\overline{G}\supset G$ and describe the subalgebras corresponding to generic points of any stratum in $\overline{G}$ as Bethe subalgebras in the Yangian of the corresponding Levi subalgebra in $\fg$. In particular, we describe explicitly all Bethe subalgebras corresponding to the closure of the maximal torus in the wonderful compactification.   
\end{abstract}
\maketitle

\section{Introduction}
\subsection{Bethe subalgebras in the Yangian}
Let $\fg$ be a complex simple  Lie algebra, $G$ the corresponding adjoint Lie group, $G^{reg}$ the set of regular elements in $G$. The Yangian $Y(\fg)$ is the canonical Hopf algebra deformation of the enveloping algebra $U(\fg[t])$ of the Lie algebra of polynomial maps $\bc \to \fg$. On the other hand, $Y(\fg)$ is a deformation of the algebra $\mathcal{O}(G_1[[t^{-1}]])$ of polynomial functions on the first congruence subgroup $G_1[[t^{-1}]]$ with the rational $r$-matrix Poisson bracket. For $\fg = \fsl_n$ this algebra appears in  the works of L. Fadeev and
St.-Petersburg school in relation to the inverse scattering method, see e.g. \cite{tarasov}, \cite{fadeev}. In full generality this algebra firstly appears in the paper of V. Drinfeld \cite{drin}.

There is a family of commutative subalgebras $B(C)$ in $Y(\fg)$, called Bethe subalgebras, parameterized by $C \in G$. In this generality such subalgebras were first mentioned by Drinfeld \cite{drin2}. For $\fg = \fsl_n$ these subalgebras were studied by M. Nazarov and G.Olshanski in \cite{nazol}, for classical Lie algebras in the work of A.Molev \cite{molev2}. In  \cite{mo} D. Maulik and A. Okounkov define Bethe subalgebras in the $RTT$ Yangian in connection with quantum cohomology of quiver varieties. We describe the associated graded of $B(C)$ in $\mathcal{O}(G_1[[t^{-1}]])$ as follows.
For any finite-dimensional representations $V$ of $G$ we assign the $\bc[[t^{-1}]]$-valued function $\sigma_V(C)\in\mathcal{O}(G_1[[t^{-1}]])$ such that $\sigma_V(C)(g(t)) := \Tr_V Cg(t)$ for all $g(t)\in G_1[[t^{-1}]]$. Then the subalgebra $\gr B(C)\subset\mathcal{O}(G_1[[t^{-1}]])$ is
generated by all Fourier components of $\sigma_V(C)$ for all finite-dimensional representations $V$ of $G$. This implies that if $C \in G^{reg}$ then $B(C)$ is a free polynomial algebra with $\rk\fg$ infinite series of generators of degrees $1,2,3,\ldots$. In particular, it has the same Poincar\'e series as the Cartan subalgebra of $Y(\fg)$ (Theorem \ref{thm1}). These results seem to be well-known, but we did not find any proof of them in the literature.

This allows to define {\em limit} Bethe subalgebras with the same Poincar\'e series, considering the closure of $G^{reg}$ in appropriate Grassmanians, see \cite{ir}. On the other hand, one can extend the family of Bethe subalgebras using the wonderful compactification of the adjoint group $G$. In the present paper we describe the connection between these approaches.

\subsection{Wonderful closure.} The main purpose of this paper is to extend the definition of Bethe subalgebras to $C\in\oG$, the De Concini--Procesi {\it wonderful compactification} of the group $G$, \cite{DCP}. The latter is a smooth projective $G\times G$-variety which contains $G$ as an open subset. The strata (i.e. $G\times G$-orbits) of $\oG$ are indexed by all subsets of simple roots of $\fg$ and can be described as $\frac{G\times G\times G_I}{P_I\times P_I^-}$ where $P_I$, $P_I^-$ are opposite parabolic subgroups determined by a subset $I$, and $G_I$ is the semisimple quotient of the Levi subgroup $L_I=P_I\cap P_I^-$. In particular, if $I = \emptyset$ then $P_I$ is a Borel subgroup of $G$, so the minimal $G\times G$-orbit in $\oG$ is just the Cartesian square of the flag variety, $G/P_{\emptyset}\times G/P_{\emptyset}$. The subalgebras corresponding to the boundary points of $\oG$ are some new commutative subalgebras in $Y(\fg)$. We describe explicitly the subalgebras corresponding to generic points of any stratum of $\oG$. Namely, these subalgebras are conjugate to Bethe subalgebras of the Yangian of some Levi subalgebra $\fl\subset\fg$ (Theorem~\ref{bethelevi}). In particular, for any $C$ in the open $G$-orbit on the closed stratum $G/P_{\emptyset}\times G/P_{\emptyset}$ the corresponding Bethe subalgebra is some conjugate of the Cartan subalgebra $H$ of $Y(\fg)$. 

We show that Bethe subalgebras corresponding to boundary points of $\oG$ are contained in some limit Bethe subalgebras. Moreover, we expect that the parameter space of limit Bethe subalgebras is some resolution of $\oG$, i.e. there is a proper birational surjective map from this parameter space to $\oG$. 



\subsection{Restriction to the maximal torus.} Our main interest is the family of Bethe subalgebras $B(C)$ such that $C$ is from the given maximal torus $T\subset G$. The action of such $B(C)$ preserves the weight decomposition of any $Y(\fg)$-module. We show that all points in the closure $\overline{T}$ of $T$ in the wonderful compactification $\overline{G}$ are ``generic'', i.e. the subalgebras corresponding to \emph{all} boundary points of $\overline{T}$ are Bethe subalgebras in the Yangians of some Levi subalgebras (Proposition~\ref{torus}). For $\fg=\fsl_n$ this gives a new proof for \cite[Theorem~5.3.1]{ir} describing limit Bethe subalgebras in the Yangian of $\fsl_n$ as the parameter goes to the infinity.

\subsection{$RTT$-presentaton of the Yangian} The most convenient way to define Bethe subalgebras is the $R$-matrix (or $RTT$) presentation of the Yangian which determines $Y(\fg)$ as the algebra generated by the evaluation of the universal $R$-matrix on some particular representation, modulo quadratic relations which follow from the Yang-Baxters equation.

We follow the recent paper of Wendlandt \cite{wend} where the description $RTT$ presentation of $Y(\fg)$ for any non-trivial representation $V$ is given. For our purpose we take $V=\bigoplus V(\omega_i, 0)$, the direct sum of the fundamental representations $V(\omega_i, 0)$ of $Y(\fg)$. In this presentation Bethe subalgebras are generated by the traces of the $T$-matrix in all fundamental representations, i.e. the generators of a Bethe subalgebra are just linear combinations of the generators of $Y(\fg)$. Our main tool for proving Theorem \ref{bethelevi} is the presentation of the Yangian of a Levi subalgebra in terms of this $RTT$ presentation (Proposition~\ref{ylevi}).

\subsection{The paper is organized as follows.}
In Section~2 we recall some known facts about Yangians and prove some technical propositions. In Section~3 we define an embedding of the Yangian of a Levi subalgebra to  the Yangian of $\fg$ in the $RTT$ presentation. In Section~4 we define Bethe subalgebras, describe their classical limits and prove that they are free polynomial algebras for any regular value of the parameter. In Section~5 we define Bethe subalgebras corresponding to boundary points in $\oG$ and describe these subalgebras for a Zariski open subset of any stratum of $\oG$. In Section~6 we compare our construction in the case $\fg=\fsl_n$ with that of \cite{ir}.

\subsection{Acknowledgements.} We thank the referees for the careful reading of the first version of the text and for many helpful remarks. This research was carried out within the HSE University Basic Research Program and funded by the Russian Academic Excellence Project '5-100'. The results of Section 4 has been obtained under support of the RSF grant 19-11-00056. The work of both authors has also been supported in part by the Simons Foundation. The first author is a Young Russian Mathematics award winner and would like to thank its sponsors and jury.


\section{Three realizations of the Yangian}
\subsection{Notation}
Let $\fg$ be a simple Lie algebra over $\bc$ of rank $n$, $\tilde G$ the corresponding connected simply-connected group, $G = \tilde G / Z(\tilde G)$ the adjoint form of a group, $\Delta = \{\alpha_1, \ldots, \alpha_n\}$ the set of simple roots of $\fg$, $\Phi^+$ the set of positive roots,  $\Phi$ the set of all roots,  $\{ \omega_i \,|\, i=1,\ldots,n \}$ the set of fundamental weights. Let $e_i, f_i, i \in \Delta$ be the standard Chevalley generators of $\fg$.

We normalize the invariant scalar product on $\fg$ such that $(\alpha_i,\alpha_i) = 2$ for short simple roots.
By $t_{\omega_i}$ we denote the element in $\fh \simeq \fh^*$ which corresponds to $\omega_i$ by means of this scalar product. Similarly, we denote by $h_i$ the element corresponds to $\alpha_i$. Let $d_i = (\alpha_i,\alpha_i)/2$. Note that $t_{\omega_i}$ and $\frac{1}{d_i} h_i$ are dual bases of $\fh$ with respect to the invariant scalar product. 

For any $\alpha \in \Phi^+$, let $x_{\alpha}^+$ and $x_{\alpha}^-$ be the elements in the corresponding root spaces of $\fg$ such that $(x_{\alpha}^+, x_{\alpha}^-) = 1$. Note that $x_{\alpha_i}^+ = -d_i^{1/2} e_i, x_{\alpha_i}^- = -d_i^{1/2} f_i$.

We define the Casimir elements by
$$c = \sum_{\alpha \in \Phi^+} (x_{\alpha}^+ x_{\alpha}^- + x_{\alpha}^- x_{\alpha}^+) + \frac{1}{d_i}\sum_i t_{\omega_i} h_i \in U(\fg),$$  $$\Omega =  \sum_{\alpha \in \Phi^+} (x_{\alpha}^+ \otimes x_{\alpha}^- + x_{\alpha}^-\otimes x_{\alpha}^+) + \frac{1}{d_i}\sum_i t_{\omega_i}\otimes h_i \in \fg\otimes \fg. $$
Note that the Casimir element can be represented as $\Omega =  \sum_a x_a \otimes x^a,$ where $x_a$ and $x^a$ are any dual bases of the Lie algebra $\fg$.

\begin{defn}
The Yangian $Y(\fg)$ is a unital associative algebra over $\bc$ generated by the elements $\{x, J(x) \, | \, x \in \fg \}$ with the following defining relations:
$$xy-yx = [x,y], \qquad J([x,y]) = [J(x), y ],  $$
$$J(cx+dy) = cJ(x) + dJ(y),$$
$$[J(x),[J(y),z]] - [x,[J(y),J(z)]]= \sum_{\lambda, \mu, \nu \in \Lambda} ([x,x_{\lambda}], [[y,x_{\mu}],[z,x_{\nu}]])\{x_{\lambda}, x_{\mu}, x_{\nu}\},$$
\begin{eqnarray*}[[J(x),J(y)],[z,J(w)]]+[[J(z),J(w)],[x,J(y)]] = \\
 =\sum_{\lambda,\mu,\nu\in \Lambda}\left(([x,x_\lambda],[[y,x_\mu],[[z,w],x_\nu]])+([z,x_\lambda],[[w,x_\mu],[[x,y],x_\nu]])\right)\{x_\lambda,x_\mu,J(x_\nu)\}\end{eqnarray*}
for all $x,y,z,w\in \fg$ and $c,d\in \bc$, where $\{x_{\lambda}\}_{\lambda \in \Lambda}$ is some orthonormal basis of $\fg$, $\{x_1,x_2,x_3\}=\frac{1}{24}\sum_{\pi \in \mathfrak{S}_3}x_{\pi(1)}x_{\pi(2)}x_{\pi(3)}$ for all $x_1,x_2,x_3\in Y(\fg)$. 
\end{defn}
Note that in the case $\fg \not\simeq \fsl_2$ the last relation can be omitted.
\begin{defn}
The Yangian $Y_{new}(\fg)$ is a unital associative algebra over $\bc$ generated by elements $\{e_i^{(r)}, f_i^{(r)}, h_i^{(r)} \, | \, i=1, \ldots, n; r \geqslant 1 \}$ with the following defining relations:

$$[h_i^{(s)}, h_j^{(s)}] = 0,$$
$${[e_i^{(r)}, f_j^{(s)}]} = \delta_{ij} h_i^{(r+s-1)},$$
$${[h_i^{(1)}, e_j^{(s)}]} = (\alpha_i, \alpha_j) e_j^{(s)},$$
$${[h_i^{(r+1)},e_j^{(s)}]} - [h_i^{(r)}, e_j^{(s+1)}] = \frac{(\alpha_i, \alpha_j)}{2} (h_i^{(r)} e_j^{(s)} + e_j^{(s)} h_i^{(r)}) ,$$
$${[h_i^{(1)}, f_j^{(s)}]} = - (\alpha_i, \alpha_j) f_j^{(s)},$$
$${[h_i^{(r+1)},f_j^{(s)}]} - [h_i^{(r)}, f_j^{(s+1)}] = -\frac{(\alpha_i, \alpha_j)}{2} (h_i^{(r)} f_j^{(s)} + f_j^{(s)} h_i^{(r)}) ,$$
$${[e_i^{(r+1)}, e_j^{(s)}]} - [e_i^{(r)}, e_j^{(s+1)}] = \frac{(\alpha_i, \alpha_j)}{2} (e_i^{(r)} e_j^{(s)} + e_j^{(s)} e_i^{(r)}),$$
$${[f_i^{(r+1)}, f_j^{(s)}]} - [f_i^{(r)}, f_j^{(s+1)}] = -\frac{(\alpha_i, \alpha_j)}{2} (f_i^{(r)} f_j^{(s)} + f_j^{(s)} f_i^{(r)}),$$
$$i \neq j, N = 1 - a_{ij} \Rightarrow
\operatorname{sym} [e_i^{(r_1)}, [e_i^{(r_2)}, \cdots
[e_i^{(r_N)}, e_j^{(s)}]\cdots]] = 0 $$
$$i \neq j, N = 1 - a_{ij} \Rightarrow
\operatorname{sym} [f_i^{(r_1)}, [f_i^{(r_2)}, \cdots
[f_i^{(r_N)}, f_j^{(s)}]\cdots]] = 0 $$

\end{defn}
We will need the following elements of $Y_{new}(\fg)$:
$$e_{\alpha_i}^{(r)} = e_i^{(r)}$$
$${[e_{\hat{\alpha}}^{(r)}, e_{\check{\alpha}}^{(1)}]} = e_{\alpha}^{(r)}$$
$$f_{\alpha_i}^{(r)} = f_i^{(r)}$$
$${[f_{\hat{\alpha}}^{(r)}, f_{\check{\alpha}}^{(1)}]} = f_{\alpha}^{(r)}$$
Here $\check\alpha$ is the smallest simple root such that $\hat\alpha = \alpha - \check\alpha$ is again a positive root.
\begin{thm}(\cite[Proposition 3.2]{kamn})
Ordered monomials in the variables $e_{\alpha}^{(r)}, h_i^{(r)}, f_{\alpha}^{(r)}$, with $\alpha\in\Phi^+,\ i\in\Delta,\ r\in\mathbb{Z}_{>0}$, form a PBW basis of $Y_{new}(\fg)$.
\end{thm}

\begin{remm}
\label{gennew}
In the course of proving Theorem 2.4 it is easy to see that the set $$e_i^{(1)}, f_i^{(1)}, h_i^{(1)}, h_i^{(2)}, i=1, \ldots, n$$ generates $Y_{new}(\fg)$.
\end{remm}

Following Drinfeld \cite{drin2} we call the subalgebra $H\subset Y(\fg)$ generated by all $h_i^{(r)}, 1 \leq i \leq n, 1 \leq r$ the \emph{Cartan subalgebra}. We will need the following maximality property of the Cartan subalgebra:

\begin{prop}
\label{cartan}
The centralizer of the set $\{ h_i^{(2)} \, | \, i =1, \ldots, n  \}$  is $H$. In particular, $H$ is a maximal commutative subalgebra of $Y_{new}(\fg)$.
\end{prop}
\begin{proof}
We have to show that everything commuting with $h_i^{(2)},\ i=1,\ldots,n,$ belongs to $H$. Any element of $Y_{new}^{(r)}$ can be represented as a sum of monomials in $e_{\alpha}^{(r)}, h_i^{(r)}, f_{\alpha}^{(r)}=e_{-\alpha}^{(r)}$. Consider the lexicographic order on the monomials determined by the following order on the PBW generators: $h_i^{(r)}<h_j^{(s)}$ if $r<s$ or $r=s$ and $i<j$; $h_i^{(r)}<e_j^{(s)}$ and $h_i^{(r)}<f_j^{(s)}$ for all $i,j,r,s$; $e_\alpha^{(r)}<e_\beta^{(s)}$ if $r<s$ or $r=s$ and $\alpha<\beta$ (for some order on $\Phi$). Suppose there is an element $A\in Y_{new}(\fg)$ which commutes with all $h_i^{(2)}$ and does not belong to $H$. Let $e_\alpha^{(r)}$ be maximal divisor of the leading monomial $lm(A)$ of $A$, and let $i$ be such that $(\alpha_i,\alpha)\ne0$. Then the leading monomial of $[h_i^{(2)},lm(A)]$ is divisible by $e_{\alpha}^{(r+1)}$ while the other monomials of $[h_i^{(2)},A]$ are not. Hence the leading monomial cannot be canceled, so the commutator is not zero.
\end{proof}

\begin{thm}(\cite{GRW})
\label{phi}
There is an isomorphism $\phi:Y_{new}(\fg)\to Y(\fg)$ such that
\begin{center}
$\phi(h_i^{(1)}) = h_i$, \qquad  $\phi(h_i^{(2)}) = J(h_i) - v_i$, \\
$\phi(e_i^{(1)}) = x_{\alpha_i^+}$, \qquad $\phi(e_i^{(2)}) = J(x_{\alpha_i}^+) - w_i^+$, \\
$\phi(f_i^{(1)}) = x_{\alpha_i^-}$, \qquad $\phi(f_i^{(2)}) = J(x_{\alpha_i}^-) - w_i^-.$
\end{center}
Here $$v_i = \dfrac{1}{4} \sum_{\alpha \in \Phi^+} (\alpha, \alpha_i) \{x_{\alpha}^+, x_{\alpha}^-\} - \dfrac{1}{2} h_i^{2}, $$
$$w_i^\pm = \pm\dfrac{1}{4}\sum_{\alpha \in \Phi^+} \{[x_i^{\pm}, x_{\alpha}^{\pm}], x_{\alpha}^{\mp}\} - \frac{1}{4}\{x_i^{\pm}, h_i\},$$
$$\{x_{\alpha}, x_{-\alpha}\} = x_{\alpha} x_{-\alpha} + x_{-\alpha} x_{\alpha}.$$

\end{thm}

Throughout the paper the Yangian $Y(\fg)$ is identified with $Y_{new}(\fg)$ via the isomorphism of Theorem 2.7.

Following Wendlandt \cite{wend} we will consider the RTT Yangian $Y_V(\fg)$ for the Lie algebra $\fg$, where $V = \bigoplus_i V(\omega_i, 0)$ is the sum of the fundamental representations of the Yangian $Y(\fg)$.

More precisely, let $\underline{\lambda} = \{\lambda_{i,r} \, | \, i = 1, \ldots, n, r \geqslant 0, \lambda_{i,r} \in \bc \}$  be a set of complex numbers, $V$ be a representation of $Y(\fg)$. Then the $\underline{\lambda}$-weight space is by definition the subspace of $V$
$$\{ v \in V \, | \, h_i^{(r)} v = \lambda_{i,r} v \  \text{for all} \ i,r \}.$$
The $Y(\fg)$-module called highest weight module, if there exist $w \in V$ such that $e_i^{(r)} \cdot w = 0$ for all $i,r$ and $V = Y(\fg) \cdot w$.
Using the technique of Verma modules one can construct the irreducible module $V(\underline{\lambda})$ of the highest weight $\underline{\lambda}$.
The following theorem is well-known.
\begin{thm}
The irreducible $Y(\fg)$-module $V(\underline{\lambda})$ is finite dimensional iff there exist polynomials $P_i(u) \in \bc[u]$ such that
$$\dfrac{P_i(u+d_i)}{P_i(u)} = 1 + \sum_{r=0}^\infty \lambda_{i,r} u^{-r-1} $$
in the sense that the right-hand side is the Laurent expansion of the left-hand side at $u = \infty$.
\end{thm}
\begin{defn}
A finite dimensional irreducible $Y(\fg)$-module is called fundamental if the associated polynomials are given by
$$P_j(u) = \begin{cases}
1, & j \ne i, \\
u - a, & j = i,
\end{cases}$$
for some $i=1, \ldots n$. We denote it by $V(\omega_i,a)$.
\end{defn}

We consider the representation $V = \bigoplus_i V(\omega_i, 0)$ of $Y(\fg)$. Let $\rho:Y(\fg)\to\End(V)$ be the homomorphism corresponding to this representation.

Note that the restriction of $V(\omega_i,0)$ to $\fg$ decomposes as
$$V(\omega_i,0) = V_{\omega_i} \oplus \bigoplus_{\mu < \omega_i} V_{\mu}^{\oplus k_{\mu}}$$
Here $V_{\mu}$ is the irreducible representation of $\fg$ of highest weight $\mu$ and $\mu < \omega_i$ means that $\omega_i - \mu$ is a sum of positive roots.

There is the unique formal series $\hat R(u) \in (Y(\fg)\otimes Y(\fg))[[u^{-1}]]$ with certain properties called the {\it universal $R$-matrix}. For the definition, see \cite[Theorem 3.4]{wend}.
Let $R(u) = (\rho \otimes \rho)\hat R(-u)$. We fix a basis of $V$ and regard $R(u)\in \End(V)^{\otimes 2}[[u^{-1}]]$ as a matrix in this basis.

\begin{defn}
\label{extyangian}
The extended Yangian $X_V(\fg)$ is a unital associative algebra generated by the elements
$t_{ij}^{(r)},  1 \leq i, j \leq \dim V; r \geq 1$  with the defining relations 
$$R(u-v) T_1(u) T_2(v) = T_2(v) T_1(u) R(u-v) \quad \text{in} \ \End(V)^{\otimes 2} \otimes X_V(\fg)[[u^{-1}, v^{-1}]].$$
Here $$T(u) = [t_{ij}(u)]_{i,j = 1, \ldots, \dim V} \in \End V \otimes X_V(\fg),$$
$$t_{ij}(u) = \delta_{ij} + \sum_r t_{ij}^{(r)} u^{-r}$$ and $T_1(u)$ (resp. $T_2(u)$) is the image of $T(u)$ in the first (resp. second) copy of $\End V$.

\end{defn}

\begin{defn}
The Yangian $Y_V(\fg)$ is a unital associative algebra generated by the elements
$t_{ij}^{(r)},  1 \leq i, j \leq \dim V; r \geq 1$  with the defining relations 
$$R(u-v) T_1(u) T_2(v) = T_2(v) T_1(u) R(u-v) \quad \text{in} \ \End(V)^{\otimes 2} \otimes Y_V(\fg)[[u^{-1}, v^{-1}]],$$
$$S^2(T(u))= T(u+\frac{1}{2}c_{\fg}),$$
where $S(T(u))=T(u)^{-1}$ is the antipode map and
$c_{\fg}$ is the value of the Casimir element of $\fg$ on the adjoint representation.

\end{defn}

\begin{thm}(see \cite{wend})
\label{iso}
There is an isomorphism $\psi:Y_V(\fg)\to Y(\fg)$ such that
$$\psi: T(u) \mapsto (\rho \otimes 1) \hat R(-u).$$
\end{thm}

\begin{thm}(see \cite[Theorem 7.3]{wend})
\label{tensor}
Extended Yangian $X_V(\fg)$ is isomorphic to the tensor product $Y_V(\fg) \otimes \bc[x_1^{(r)}, \ldots, x_k^{(r)}]$ where $r \in \bz_{>0}$, $k = \dim \End_{Y(\fg)} V$.
\end{thm}

\begin{rem}
According to \cite{wend}, to define the RTT Yangian $Y_V(\fg)$, one can use any finite-dimensional representation $V$ of $Y(\fg)$ which is not a sum of trivial representations, and Theorems 2.12 and 2.13 are still valid.
\end{rem}

\subsection{The algebra $\mathcal{O}(G_1[[t^{-1}]])$}

Note that since $G$ is a complex simple Lie group, both $G$ and $\tilde G$ are complex algebraic groups.
 
\begin{defn}
Let $V$ be a representation of $\tilde G$, $w \in V, \beta \in V^*$. The matrix entry is the function on $\tilde G$ given by $\Delta_{\beta,w}(g) = \beta(g \cdot w)$.
\end{defn}

According to Peter--Weyl theorem, the ring $\mathcal{O}(\tilde G)$ is spanned by matrix entries of all irreducible finite dimensional representations of $\tilde G$.

Let $G_1[[t^{-1}]]$ be the first congruence subgroup of the group $G[[t^{-1}]]$ of $\bc[[t]]$-points of $G$, i.e. the kernel of the homomorphism $G[[t^{-1}]]\to G$. Note that $G_1[[t^{-1}]]$ depends only on the formal neighborhood of $e\in G$. In particular the map $\tilde G \to G$ induces an isomorphism $\tilde G_1[[t^{-1}]] \tilde{\to} G_1[[t^{-1}]]$ .  

Let $\mathcal{O}(G_1[[t^{-1}]])$ be the ring of complex valued (polynomial) functions on $G_1[[t^{-1}]]$. From Peter-Weyl Theorem we see that the ring $\mathcal{O}(G_1[[t^{-1}]])$ is generated by Fourier components of the corresponding matrix entries of  $G_1[[t^{-1}]])$. More precisely, let $V$ be a representation of $\tilde G$. To any $w \in V$ and $\beta \in V^{*}$ we assign the $\bc$-valued functions $\Delta_{\beta,w}^{(s)}$ on $G_1[[t^{-1}]]$ determined as $\langle\beta,g(t)\cdot w\rangle=1+\sum\limits_{s=1}^\infty\Delta_{\beta,w}^{(s)}(g(t))t^{-s}$. Then the elements $\Delta_{\beta,w}^{(s)}$ generate $\mathcal{O}(G_1[[t^{-1}]])$. 

\subsection{Filtration on $Y_{new}(\fg)$}
We define a filtration on $Y_{new}(\fg)$ by
$$\deg e_{\alpha}^{(r)} = \deg f_{\alpha}^{(r)} = \deg h_i^{(r)} = r.$$
By $Y_{new}^{(r)}$ we denote the $r$-th filtered component of $Y_{new}(\fg)$.
Note that from Theorem 2.4 it follows that this is indeed a filtration and $\gr Y_{new}(\fg)$ is commutative.

Let $v_i$ be the highest weight vector in $V_{\omega_i}$  and $w_i = f_i v_i$ be the ``pre-highest'' weight vector. Note that $v_i$ and $w_i$ are unique up to proportionality since the corresponding weight spaces are one-dimensional. Let $v_i^*, w_i^*\in V_{\omega_i}^*$ be the weight elements such that $\langle v_i^*,v_i\rangle=1$ and $\langle w_i^*,w_i\rangle=1$.
Consider $V_{\omega_i}$ as a representation of $\tilde G$.
Following \cite{kamn}, we define $ \Delta_{\omega_i, \omega_i}(u):=\Delta_{v_i^*, v_i}(u)$, $\Delta_{\omega_i, s_i \omega_i}(u):=\Delta_{w_i^*, v_i}(u)$ and $\Delta_{s_i\omega_i, \omega_i}(u):=\Delta_{v_i^*, w_i}(u)$. 
We consider the following generating series
$$H_i(u) = 1 + \sum_{r \geq 1} h_i^{(r)} u^{-r}, F_i(u) = \sum_{r \geq 1} f_i^{(r)} u^{-r}, E_i(u) = \sum_{r \geq 1 }e_i^{(r)} u^{-r}\in Y_{new}(\fg)[[u^{-1}]].$$

\begin{thm}(\cite[Theorem 5.13]{comult}, \cite{kamn})
\label{poison}
There is an isomorphism of Poisson algebras $\theta: \gr Y_{new}(\fg) \simeq \mathcal{O}(G_1[[t^{-1}]])$ such that
$$\theta(H_i(u)) = \prod_j \Delta_{\omega_j, \omega_j}(u)^{-a_{ji}}$$
$$\theta(F_i(u)) = d_i^{-1/2} \dfrac{\Delta_{\omega_i, s_i \omega_i}(u)}{\Delta_{\omega_i, \omega_i}(u)}$$
$$\theta(E_i(u)) = d_i^{-1/2} \dfrac{\Delta_{s_i \omega_i, \omega_i}(u)}{\Delta_{\omega_i, \omega_i}(u)}$$

\end{thm}



\begin{cor}\label{pgen}
The Poisson algebra $\mathcal{O}(G_1[[t^{-1}]])$ is generated by $\Delta_{\omega_i, \omega_i}^{(1)}$, $\Delta_{\omega_i, s_i\omega_i}^{(1)}$, $\Delta_{s_i\omega_i, \omega_i}^{(1)}$ and $\Delta_{\omega_i, \omega_i}^{(2)}$. Here superscripts $(1)$ and $(2)$ indicate taking the coefficients at $u^{-1}$ and $u^{-2}$ in the formal power series expansion.
\end{cor}

\begin{proof}
Indeed, according to the Theorem \ref{poison}, the images of Drinfeld's new generators $h_i^{(1)}$, $e_i^{(1)}$, $f_i^{(1)}$ under $\theta$ are linear combinations of $\Delta_{\omega_j, \omega_j}^{(1)}$, $\Delta_{\omega_j, s_j\omega_j}^{(1)}$, $\Delta_{s_j\omega_j, \omega_j}^{(1)}$. Moreover, $\theta(h_i^{(2)})$ is a linear combination of $\Delta_{\omega_j, \omega_j}^{(2)}$ plus a quadratic-linear expression of $\Delta_{\omega_j, \omega_j}^{(1)}$. On the other hand, according to Remark~\ref{gennew} the images of $h_i^{(1)}$, $e_i^{(1)}$, $f_i^{(1)}$ and $h_i^{(2)}$ in $\gr Y_{new}(\fg)$ generate it as a Poisson algebra.  
\end{proof}

\subsection{Filtration on $Y_{V}(\fg)$} Let us define a filtration on $Y_V(\fg)$ setting
$$\deg \, t_{ij}^{(r)} = r.$$
By $Y_V^{(r)}$ we denote the $r$-th filtered component of $Y_V(\fg)$.
It is easy to see that $\gr Y_V(\fg)$ with respect to this grading is commutative.

Recall that the Casimir element can be represented as $\Omega =  \sum_a x_a \otimes x^a,$ where $x_a$ and $x^a$ are some dual bases of the Lie algebra $\fg$.
Let $e_i, i = 1, \ldots, \dim V$ and $e_i^*, i = 1, \ldots, \dim V $ be an arbitrary basis of $V$ and $V^{*}$ respectively. Let $\beta = \sum a_{i} e_i^*, w = \sum b_j e_j$ and put by definition
$$t_{\beta,w}(u) = \sum_i \sum_j a_i b_j t_{ij}(u).$$
By definition, put
$$\Omega_{\rho} = (\rho \otimes \rho)\Omega$$
Recall that $R(u) = Id - \Omega_{\rho}u^{-1} + \sum_{k \geqslant 2} R^{(k)} u^{-k}$.
\begin{lem}
\label{poisson}
$\gr Y_V(\fg)$ is a Poisson algebra with the Poisson structure defined by the following formula

$$\{t_{\beta_1 v_1}(u),t_{\beta_2 v_2}(v)\} = \dfrac{1}{v-u} \sum_a \left(t_{\beta_1, x_a v_1}(u) t_{\beta_2, x^a v_2}(v) - t_{x_a \beta_1, v_1}(u) t_{x^a \beta_2, v_2}(v)\right)$$
\end{lem}
\begin{proof}
Let us rewrite the defining $RTT$ relations in the following way:
\begin{eqnarray*}[T_1(u), T_2(v)] = \dfrac{1}{u-v}\left(\Omega_{\rho} T_1(u) T_2(v) - T_2(v)T_1(u) \Omega_{\rho}\right)  \\ +  \sum_{k \geqslant 2} \dfrac{1}{(u-v)^k} \left(T_2(v)T_1(u) R^{(k)} - R^{(k)} T_1(u) T_2(v)\right).
\end{eqnarray*}
Note that only the first summand gives a nonzero contribution into the Poisson bracket. 
Comparing the evaluation of LHS and RHS on $\beta_1,v_1,\beta_2,v_2$ we get the assertion.
\end{proof}

\subsection{Relation between two filtrations on $Y(\fg)$}
Let us fix $\omega_i$ and choose a basis of $V_{\omega_i}$ such that first basis vector is $v_i$ and second is $w_i$.
We use the notation $t_{11}^{(1)} = t_{v_i^* v_i}^{(1)}, t_{12}^{(1)} = t_{v_i^* w_i}^{(1)}, t_{21}^{(1)} = t_{w_i^* v_i}^{(1)}, t_{11}^{(2)} = t_{v_i^* v_i}^{(2)}$ in what follows. Slightly abusing notations, we denote the symbols of $t_{ij}^{(r)}$ in the associated graded $\gr Y_V(\fg)$ by the same letters $t_{ij}^{(r)}$. 
Here again superscripts $(1)$ and $(2)$ indicate taking the coefficient at $u^{-1}$ and $u^{-2}$ in the power series $t_{ij}(u)$.
We denote $\Delta_{\omega_i, \omega_i}(u)$ by $\Delta_{11}(u)$,
$\Delta_{\omega_i, s_i \omega_i}(u)$ by $\Delta_{12}(u)$ and $\Delta_{s_i\omega_i, \omega_i}(u)$ by $\Delta_{21}(u)$.

\begin{lem}
\label{t11}
For any $\omega_i$
the images of $t_{11}^{(1)}, t_{12}^{(1)}, t_{21}^{(1)}, t_{11}^{(2)}$ under the map $\gr (\phi^{-1} \circ \psi)$ in $\gr Y_{new}(\fg)$ are $\Delta_{11}^{(1)}, \Delta_{12}^{(1)}, \Delta_{21}^{(1)}, \Delta_{11}^{(2)}$ respectively.
\end{lem}

\begin{proof} 
We compute the composition $\phi^{-1} \circ \psi$, see Theorem \ref{phi} and Theorem \ref{iso}.

Firstly, $t_{11}^{(1)}$ maps to the coefficient of $u^{-1}$ in the in the $(1,1)$-th entry of $(\rho\otimes 1)\hat R(-u)$ i.e. to the $(1,1)$-th entry of $$-\sum_{\alpha \in \Phi} \rho(x_{\alpha}^+) \otimes x_{\alpha}^- -  \sum_i \dfrac{1}{d_i} \rho(h_i) \otimes t_{\omega_i} \in \End(V) \otimes Y(\fg)$$
We see that $t_{11}^{(1)}$ goes to $-t_{\omega_i}^{(1)}$.
Here $t_{\omega_i}^{(1)} = \sum_j b_{ij} h_j^{(1)}$ and $(b_{ji})_{i,j=1, \ldots, n}$ are the coefficients of matrix inverse to Cartan matrix.

Secondly, $t_{12}^{(1)}$ maps to the coefficient of $u^{-1}$ in the $(1,2)$-th entry in $(\rho\otimes 1)\hat R(-u)$, i.e. to the $(1,2)$-th entry of
$$-\sum_{\alpha \in \Phi^+} \rho(x_{\alpha}^{+}) \otimes x_{\alpha}^{-} - \sum_{\alpha \in \Phi^{+}} \rho(x_{\alpha}^{-}) \otimes x_{\alpha}^{+} -  \dfrac{1}{d_i} \sum_i \rho(h_i) \otimes t_{\omega_i} \in \End(V) \otimes Y(\fg)$$
We see that $t_{12}^{(1)}$ maps to $d_i^{1/2} x_{\alpha_i}^- = d_i^{1/2} f_{i}^{(1)}$. (Here we use the fact that the representation of $sl_2 = \{e_i,f_i,\frac{1}{d_i}h_i\}$ generated from the highest vector is 2-dimensional).
Similarly, $t_{21}^{(1)}$ maps to $x_{\alpha_i}^+ = d_i^{1/2} e_{i}^{(1)}$.

Consider the element $t_{11}^{(2)}$. According to the expression of the second Fourier coefficient of $\hat R(-u)$ (see Theorem~3 of \cite{drin} and Theorem 3.4 of \cite{wend}) $t_{11}^{(2)}$ maps to the $(1,1)$ entry of

\begin{equation}\label{R2} -\sum_{\lambda \in \Phi^+} (\rho(x_{\alpha}^{\pm}) \otimes J(x_{\alpha}^\mp) - \rho(J(x_{\alpha}^\pm)) \otimes x_{\alpha}^\mp) - \frac{1}{d_i}\sum_{i=1}^n (\rho(h_i) \otimes J(t_{\omega_i}) - \rho(J(h_i)) \otimes t_{\omega_i}) + \frac{1}{2} (\rho \otimes 1)\Omega^2
\end{equation}

which is, up to degree $1$ terms, 

$$-J(t_{\omega_i}) + \frac{1}{2} \sum_{\alpha \in \Phi^+} \omega_i(h_{\alpha}) x_{\alpha}^- x_{\alpha}^+ + \frac{1}{2} t_{\omega_i}^2 = $$
\begin{eqnarray*} = -\sum_{j} b_{ij} h_j^{(2)} - \frac{1}{4} \sum_j \left(b_{ij}\left(\sum_{\alpha \in \Phi^+} (\alpha, \alpha_j) (x_{\alpha}^+ x_{\alpha}^- + x_{\alpha}^- x_{\alpha}^+)\right) - 2b_{ij}(h_j^{(1)})^2\right) + \\ + \frac{1}{2} \sum_{\alpha \in \Phi^+}\sum_j b_{ij}(\alpha, \alpha_j) x_{\alpha}^- x_{\alpha}^+ + \frac{1}{2}t_{\omega_i}^2 \end{eqnarray*}
Hence up to degree $1$ terms we have
\begin{eqnarray*}t_{11}^{(2)} \mapsto -\sum_{j} b_{ij} h_j^{(2)} + \frac{1}{2}\sum_j b_{ij} (h_j^{(1)})^2 + \frac{1}{2}\left(\sum_j b_{ij} h_j^{(1)}\right)^2 = \\ = -\sum_{j} b_{ij} h_j^{(2)} + \sum_j \frac{b_{ij}(b_{ij}+1)}{2} (h_j^{(1)})^2 + \sum_{j,l} b_{ij}b_{il} h_j^{(1)} h_l^{(1)}  \end{eqnarray*}
By Theorem~\ref{poison} the latter is $\Delta_{11}^{(2)}$, so we see that $t_{11}^{(2)}$ goes to $\Delta_{11}^{(2)}$.

\end{proof}

\begin{prop}
\label{rttgen}
The Yangian $Y_V(\fg)$ is generated by $t_{11}^{(1)}, t_{12}^{(1)}, t_{21}^{(1)}, t_{11}^{(2)}$ for all $\omega_i$.
\end{prop}
\begin{proof}
This follows from the proof of the previous Lemma and the fact that $Y_{new}(\fg)$ is generated by $e_i^{(1)}, f_i^{(1)}, h_i^{(1)}, h_i^{(2)}, i=1, \ldots, n$.
\end{proof}

\begin{prop}
\label{filtr}
$\phi^{-1}\circ\psi:Y_V(\fg)\to Y_{new}(\fg)$ is an isomorphism of filtered algebras.
\end{prop}
\begin{proof}
The strategy of proof is as follows:
first, we prove the easy part of PBW theorem, that is $\dim Y_V^{(r)} \leq \dim Y_{new}^{(r)}$. Next, we show that the isomorphism of Theorem \ref{iso} is a filtered algebra homomorphism and is surjective on filtered components.

Following ~\cite{wend} let us decompose $\End V$ as a representation of $\fg$:
$$\End V = \fg \oplus \mathcal{E} \oplus \mathcal{E}_c \oplus \mathcal{F}.$$
Here $\fg$ is the copy of $\fg$ in $\End V$ corresponding to the $\fg$-action on $V$, 
$\mathcal{E} = \End_{Y(\fg)} V \subset \End_{\fg} V = \mathcal{E} \oplus \mathcal{E}_c$. $\mathcal{E}_c$ is a $\fg$-invariant complement to $\mathcal{E}$ in $\End_{\fg} V$ and $\mathcal{F}$ is a $\fg$-invariant complement to $\fg \oplus \End_{\fg} V$ in $\End V$. We put $\mathcal{Z} = \mathcal{E} \oplus \mathcal{E}_c$.
Denote by $\mathcal{T}^{(r)}$ the subspace spanned by all $t_{ij}^{(r)} \in Y_V(\fg)$.
Note that as there is an isomorphism of vector spaces $\mathcal{T}^{(r)}\simeq\End V$ such that the $r$-rh component of $T(u)$, $T^{(r)} = \sum_{ij} t_{ij}^{(r)} \otimes e_{ij}$, is the identity operator in $\End(V)\otimes\End(V)$ (here we use the isomorphism $\End V \simeq (\End V)^{*}$). Hence we have the decomposition $\mathcal{T}^{(r)} = \fg^{(r)} + \mathcal{Z}^{(r)} + \mathcal{F}^{(r)}$. This gives $$Y^{(r)}_V = \fg^{(r)} + \mathcal{Z}^{(r)} + \mathcal{F}^{(r)} +  \text{expressions of degree  $\leqslant r$ in elements of $Y_V^{(r-1)}$.} $$

It follows from ~\cite[Theorem 6.5]{wend} that $Y^{(1)}_V = \fg^{(1)} \simeq \fg$. It follows from the defining relations that there is a surjective homomorphism $\mathcal{T}^{(1)}\to \mathcal{T}^{(r)}$ as $\fg^{(1)}$-modules for any $r \geqslant 1$.
Therefore, from $[\fg^{(1)}, \mathcal{F}^{(1)}] = \mathcal{F}^{(1)}$ it follows that $[\fg^{(1)}, \mathcal{F}^{(r)}] = \mathcal{F}^{(r)}$.

We claim that 
$[\fg^{(1)}, \mathcal{F}^{(r)}] = \mathcal{F}^{(r)} \subset Y_V^{(r-1)}$. 
Indeed, $T^{(r)}$ corresponds to the identity element in $\mathcal{T}^{(r)} \otimes \End V$. Hence
$$T^{(r)} \in \fg^{(r)} \otimes \fg + \mathcal{Z}^{(r)} \otimes \mathcal{Z} + \mathcal{F}^{(r)} \otimes \mathcal{F}.$$

Observe that
$$[\mathcal{F}_1^{(r)}, \fg^{(1)}_2] = [\Omega_{\rho}, \mathcal{F}_2^{(r)}] + \ \text{smaller degree terms}.$$ \ 
Using the same argument as in ~\cite[Lemma 4.2]{wend} we get $[\fg^{(1)}, \mathcal{F}^{(r)}] = \mathcal{F}^{(r)} \subset Y_V^{(r-1)}$.

Now we claim that $\mathcal{E}_c = 0$ in our associated graded algebra.
It follows from relation \cite[(5.10)]{wend} (which is still true for our filtration) and from \cite[Lemma 4.13]{wend}.
From \cite[Proposition 5.6]{wend} it follows, that the elements of degree $r$ from $\mathcal{E}$ can be represented as sums of monomials in $t_{ij}^{(k)}$ with $k< r$ of total degree non-greater than $r$. 
Finally we conclude that ordered monomials in $\fg^{(r)}, r \geq 1$ span $Y_V(\fg)$ therefore $\dim Y_V^{(r)} \leqslant \dim Y_{new}^{(r)}$.


From the proof of previous proposition we know that $t_{11}^{(1)} \mapsto -t_{\omega_i}^{(1)}, t_{12}^{(1)} \mapsto -e_{\alpha_i}^{(1)}, t_{21}^{(1)} \mapsto -f_{\alpha_i}^{(1)}$.
Note that these elements generate $Y_{new}^{(1)}$, consequently we have a surjective homomorphism from $Y_V^{(1)}$ to $Y_{new}^{(1)}$.
Define $t_{\omega_i}^{(2)} = \sum_j a_{ji} h_i^{(2)}$.
We know that $$t_{11}^{(2)} \mapsto t_{\omega_i}^{(2)} + \text{quadratic-linear expression in $Y_{new}^{(1)}$}.$$
Hence for any $i$, the element $t_{\omega_i}^{(2)}$ is in the image of $Y_V^{(2)}$, so the same is true for $h_i^{(2)}$.
Note that one can obtain the whole $Y_{new}^{(2)}$ from $h_i^{(2)}$ and $e_i^{(1)}, f_i^{(1)}$ by using commutators. 
On the other hand we know also that the degree of a commutator of elements of degrees $p$ and $q$ is not greater than $p+q-1$ (both in $Y_V$ and $Y_{new}$) hence $Y_V^{(2)}$ maps to $Y_{new}^{(2)}$ surjectively. Comparing the dimensions we see that it is in fact and isomorphism. Proceeding by induction by taking commutator with $h_i^{(2)}$ on each step we obtain the result.
\end{proof}

\begin{cor}
$\gr Y_V(\fg) = \gr Y_{new}(\fg)=\mathcal{O}(G_1[[t^{-1}]])$.
\end{cor}

\begin{prop}
\label{image}
We have $\gr \phi^{-1}\circ\psi(t_{\beta, w}^{(r)})=\Delta_{\beta, w}^{(r)}$ in $\gr Y_{new}(\fg)=\mathcal{O}(G_1[[t^{-1}]])$.
\end{prop}
\begin{proof}
From Lemma \ref{poisson} and \cite[Proposition 2.13]{kamn}
we see that the Poisson bracket of $t_{\beta,v}^{(r)}$ in $\gr Y_{V}(\fg)$ is given by the same formulas as that of $\Delta_{\beta,w}^{(r)}$ in $\mathcal{O}(G_1[[t^{-1}]])$.
Moreover, from Lemma \ref{t11} we see that the assertion of the Proposition is true for the first and second Fourier components of the matrix elements corresponding to the highest and pre-highest vectors. Namely, we have
$$\phi^{-1}\circ\psi(t_{v_i^*, v_i}^{(1)})=\Delta_{v_i^*, v_i}^{(1)};$$
$$\phi^{-1}\circ\psi(t_{v_i^*, w_i}^{(1)})=\Delta_{v_i^*, w_i}^{(1)};$$
$$\phi^{-1}\circ\psi(t_{w_i^*, v_i}^{(1)})=\Delta_{w_i^*, v_i}^{(1)};$$
$$\phi^{-1}\circ\psi(t_{v_i^*, v_i}^{(2)})=\Delta_{v_i^*, v_i}^{(2)}.$$
On the other hand by Corollary~\ref{pgen} the elements $\Delta_{v_i^*, v_i}^{(1)}$, $\Delta_{v_i^*, w_i}^{(1)}$, $\Delta_{w_i^*, v_i}^{(1)}$ and $\Delta_{v_i^*, v_i}^{(2)}$ generate $\gr Y_{new}(\fg)$, so we are done.
\end{proof}

\subsection{ABC generators.}

The following Lemma is from \cite{abc}.
\begin{lem}
Any generation series $H_i(u)$ can be
represented in the following form
\hspace{2cm}
$$H_i(u)=\frac{\prod\limits_{s\neq i}\prod\limits_{r=1}^{-a_{si}}
A_s\big(u-\frac{1}{2}(\alpha_i+r\alpha_s,\alpha_s)\big)}
{A_i(u)A_i(u-\frac{1}{2}(\alpha_i,\alpha_i))}\,,$$\ \ \ where $A_i(u)$, $i=1,\ldots,n$ are  formal series
$A_i
(u)=1+\sum\limits_{r=1}^{\infty}a_i^{(r)}u^{-r}$.
\end{lem}
We define $$B_i(u) = d_i^{1/2} A_i(u) E_i(u) = \sum\limits_{r=1}^{\infty}b_i^{(r)}u^{-r},$$ $$C_i(u) = d_i^{1/2} F_i(u) A_i(u) = \sum\limits_{r=1}^{\infty}c_i^{(r)}u^{-r}.$$

According to \cite{abc} we have
\begin{gather}
\label{abc}
    [A_j(u), B_i(v)] = \delta_{ij} \frac{d_i}{u-v} \left(B_i(u)A_i(v) - B_i(v) A_i(u)\right) \\
    \label{abc2}
    [A_j(u), C_i(v)] = \delta_{ij} \frac{d_i}{u-v} \left(C_i(u)A_i(v) - C_i(v) A_i(u)\right),
\end{gather}
We refer the reader to \cite{abc} for more details on $A_i(u), B_i(u), C_i(u)$.

We have the following non-commutative analog of Lemma \ref{image}.
\begin{prop}
\label{lift}
For any $\omega_i$ we have \\
$\phi^{-1}\circ\psi(t_{11}(u)) = A_i(u-d_i), \\ \phi^{-1}\circ\psi(t_{12}(u)) = B_i(u-d_i), \\
\phi^{-1}\circ\psi(t_{21}(u)) = C_i(u-d_i).$
\end{prop}

\begin{proof}
Let $a_i^{(k)}, b_i^{(k)}, c_i^{(k)}$ be $k$-th coefficient of the series $A_i(u), B_i(u), C_i(u)$ respectively and let $\tilde a_i^{(k)}, \tilde b_i^{(k)}, \tilde c_i^{(k)}$ be $k$-th coefficient of the series $A_i(u-d_i), B_i(u-d_i), C_i(u-d_i)$ respectively.
We have
$$t_{11}^{(1)} \mapsto a_i^{(1)} = \tilde a_i^{(1)}, \ t_{12}^{(1)} \mapsto b_i^{(1)} = \tilde b_i^{(1)}, \ t_{21}^{(1)} \mapsto c_i^{(1)} = \tilde c_i^{(1)}.$$
Now let us compute the image of $t_{11}^{(2)}$ by taking into account the degree 1 terms in the computation of Lemma~\ref{t11}. From~(\ref{R2}) and Theorem~\ref{phi} we have
\begin{eqnarray*}\ \phi^{-1}\circ\psi(t_{11}^{(2)}) = 
-\sum_{j} b_{ij} h_j^{(2)} - \frac{1}{4} \sum_j \left(b_{ij}\left(\sum_{\alpha \in \Phi^+} (\alpha, \alpha_j) (x_{\alpha}^+ x_{\alpha}^- + x_{\alpha}^- x_{\alpha}^+)\right) - 2b_{ij}(h_j^{(1)})^2\right) + \\ + (\frac{1}{4} \sum_{\alpha \in \Phi^+} \frac{1}{d_i}(\alpha,\alpha_i)(\alpha,\omega_i) +\frac{1}{2}d_i) t_{\omega_i}+ \frac{1}{2} \sum_{\alpha \in \Phi^+}\sum_j b_{ij}(\alpha, \alpha_j) x_{\alpha}^- x_{\alpha}^+ + \frac{1}{2}t_{\omega_i}^2=\\=
-\sum_{j} b_{ij} h_j^{(2)} -\frac{1}{4} \sum_j b_{ij}\left(\sum_{\alpha \in \Phi^+} (\alpha, \alpha_j) (x_{\alpha}^+ x_{\alpha}^- - x_{\alpha}^- x_{\alpha}^+)\right) + \sum_j 2b_{ij}(h_j^{(1)})^2 + \\ + \frac{1}{4} \left(\sum_{\alpha \in \Phi^+} \frac{1}{d_i}(\alpha,\alpha_i)(\alpha,\omega_i) +\frac{1}{2}d_i \right) t_{\omega_i}+ \frac{1}{2}t_{\omega_i}^2 =\\= -\sum_{j} b_{ij} h_j^{(2)} - \frac{h^{\vee}}{4}t_{\omega_i} + \dfrac{1}{2}\sum_j b_{ij}(h_j^{(1)})^2 + \left(\frac{h^{\vee}}{4} +\frac{1}{2}d_i\right) t_{\omega_i} + \frac{1}{2}t_{\omega_i}^2=\\=-\sum_{j} b_{ij} h_j^{(2)} + \dfrac{1}{2}\sum_j b_{ij}(h_j^{(1)})^2 + \frac{1}{2}d_i t_{\omega_i} + \frac{1}{2}t_{\omega_i}^2.
\end{eqnarray*}
Here $h^{\vee} \in \bz_{>0}$ is the dual Coxeter number of $\fg$.
Let $\tilde a_i^{(k)}$ be the $k$-th coefficient of the series $\tilde A_i(u) = A(u-d_i)$.
Analogously $\tilde b_i^{(k)}$ be the $k$-th coefficient of the series $\tilde B_i(u) = B(u-d_i)$.
Then $\tilde b_i^{(1)} = e_i^{(1)}$ and $\tilde b_i^{(2)} = b_i^{(2)} + d_i b_i^{(1)} = e_i^{(2)} + a_i^{(1)} e_i^{(1)} + d_i e_i^{(1)}$.

We claim that the element $\phi^{-1}\circ\psi(t_{11}^{(2)})$ satisfies the commutation relation  $$[\phi^{-1}\circ\psi(t_{11}^{(2)}),\tilde b_j^{(1)}]=-\delta_{ij}d_i \tilde b_i^{(2)},
$$
and it is the only element from the Cartan subalgebra $H$ with this property is $\tilde a_i^{(2)}$. 
Indeed, we have the following commutation relations
$$
[h_i^{(2)}, e_j^{(1)}] = \frac{(\alpha_i,\alpha_j)}{2}(h_i^{(1)} e_j^{(1)} + e_j^{(1)} h_i^{(1)}) + (\alpha_i,\alpha_j)e_j^{(2)} 
$$
$$[(h_j^{(1)})^2, e_i^{(1)}] = (\alpha_i,\alpha_j)(h_j^{(1)} e_i^{(1)} + e_i^{(1)}h_j^{(1)})$$
$$[t_{\omega_i}, e_j^{(1)}] = \delta_{ij} d_i e_i^{(1)}$$
$$[t_{\omega_i}^2, e_j^{(1)}] = \delta_{ij} d_i (t_{\omega_i} e_i^{(1)} + e_i^{(1)} t_{\omega_i})$$

Using this we have
$$[\phi^{-1}\circ\psi(t_{11}^{(2)}),\tilde b_j^{(1)}] = -\delta_{ij} d_i( e_i^{(2)} + a_i^{(1)} e_i^{(1)} + d_ie_i^{(1)} )$$

Hence we have  
$$\ t_{11}^{(2)} \mapsto \tilde a_i^{(2)}.$$

The commutation relations between 
$t_{11}(u)$ and $t_{12}(v)$ (resp. $t_{21}(v)$) are the same as between $A_i(u-d_i)$ and $B_i(v-d_i)$ (resp. $C_i(v-d_i)$). Using induction on $r$ by taking commutator with $\tilde a_i^{(2)}$ on each step we prove that $t_{12}^{(r)}$ and $t_{21}^{(r)}$ go to the $r$-th coefficients of the series $B_i(u-d_i)$ and $C_i(u-d_i)$, respectively. 

Now it remains to show that $t_{11}^{(r)}$ maps to the $r$-th Fourier coefficient of $A_i(u-d_i)$. The image of $t_{11}(u)$ lies in the Cartan subalgebra $H$, see Proposition \ref{cartan}. Moreover, we have 
$[\phi^{-1}\circ\psi(t_{11}(u)), \tilde b_i^{(1)}] = \delta_{ij} B_i(u-d_i)$. Since $H \cap \bigcap_i \left( \ker {\rm ad}\, \tilde b_i^{(1)} \right) = \{0\}$, this implies that $\phi^{-1}\circ\psi(t_{11}(u))=A_i(u-d_i)$. 



\end{proof}

\begin{rem}
We thank Alex Weeks for pointing out the mistake in the original statement of Proposition, namely, the shift by $d_i$ was missed. For another proof of Proposition, see \cite{fww}.
\end{rem}

In \cite{kamn2} the authors define the elements $T_{\beta,w}(u)\in Y_{new}(\fg)[[u^{-1}]]$ which are uniquely determined by the following property $T_{v_{i},v_{i}}(u) = A_i(u)$ and
\begin{align*}
[b_j^{(1)}, T_{\beta, \gamma}(u) ] = T_{x_{\alpha_j}^- \beta,\gamma}(u) - T_{\beta, x_{\alpha_j}^+\gamma}(u) \\
[c_j^{(1)}, T_{\beta, \gamma}(u)] = T_{ x_{\alpha_j}^+ \beta, \gamma}(u) - T_{\beta,  x_{\alpha_j}^-\gamma}(u)
\end{align*}
for all $i, j\in \{1, \ldots, n\}$.
We have the following corollary from Proposition~\ref{lift} 
\begin{cor} $\phi^{-1}\circ\psi(t_{\beta,w}(u))=T_{\beta,w}(u-d_i)$.
\end{cor}

\section{Yangian of Levi subalgebra} \label{levisub}
Let $Q$ be the Dynkin diagram of $\fg$. To any Dynkin subdiagram $I\subset Q$ we assign a Levi subalgebra $\fl$ generated by the Cartan subalgebra and the simple root elements corresponding to $I$. Denote by $\Phi^+_I$ the corresponding set of positive roots.
It is easy to see that $\fl$ is a reductive Lie algebra. We denote by $\fg_I$ the semisimple part of $\fl$, i.e. $\fg_I= [\fl,\fl]$.

We define the Yangian of $\fl$ as the subalgebra of $Y_{new}(\fg)$  generated by $e_i^{(r)}, f_i^{(r)}$ for $i \in I, r \in \mathbb{Z}_{\geqslant 0}$ and  $h_i^{(r)}, r \geq 1, i=1, \ldots ,n$. We denote this subalgebra by $Y_{new}(\fl)$. 

\begin{prop} $Y_{new}(\fl)$ is the tensor product $Y_{new}(\fg_I)\otimes \bc[a_j^{(r)}]_{j\in \Delta \setminus I, r\in\bz_{>0}}$ such that $\deg a_j^{(r)} = r$.
\end{prop}

\begin{proof}
Indeed, we can replace the generators $H_j(u)$ by $A_j(u)$. The Fourier coefficients of $A_j(u)$, $j\in\Delta\backslash I$ commute with $e_i^{(r)}$, $f_i^{(r)}$ for all $i\in I$.
\end{proof}

Let us describe the Yangian of $\fl$ in the $RTT$-realization i.e. as a subalgebra $Y_V(\fl)\subset Y_V(\fg)$. 
Note that $\fl = \fz_\fg(x)$ for $x = \sum_{j \in \Delta \backslash I} a_j^{(1)}$ and $x$ is a central element of $Y_{new}(\fl)$. Using the element $x$ one can decompose $V$ into the sum of $Y_{new}(\fl)$-submodules $V = V_I \oplus W$ in the following way. Let $\lambda$ be the eigenvalue of $x$ on the highest vector $v_i$ of $V(\omega_i,0)$ .  Let $V_I^i$ be the eigenspace of $x$ with the eigenvalue $\lambda$. Note that $V_I^i \subset V(\omega_i,0)$ is the subrepresentation generated by $Y_{new}(\fl)$ from the highest vector of $V(\omega_i,0)$: indeed, $V(\omega_i,0)$ is generated from the highest vector $v_i$ by the action of the lowering operators $f_\alpha^{(r)}$, so, since $\alpha(x) \geq 0$ for any $\alpha \in \Phi^+$ and $[x,f_\alpha^{(r)}]=-\alpha(x)f_\alpha^{(r)}$, the eigenvalues of $x$ on $V(\omega_i,0)$ belong to $\lambda-\mathbb{Z}_{\geq0}$. Moreover, the eigenvalue $\lambda$ is reached only on $Y_{new}(\fl)v_i$, since $\alpha(x) > 0$ for any $\alpha \in \Phi^+ \backslash \Phi^+_I$. Clearly, $V_I^i$ is a direct summand in $V(\omega_i,0)$ i.e. we have a decomposition $V(\omega_i,0) = V_I^i \oplus W_i$ as $Y_{new}(\fl)$-modules. We put $V_I = \oplus_i V_I^i$ and $W = \oplus_i W_i$. 

\begin{prop}
$V_I^i$ for $i \in I$ is the fundamental representation $V_I(\omega_i,0)$ of $Y_{new}([\fl,\fl])$.
\end{prop}
\begin{proof}
Note that the highest weight for $V_I^i$ is the same as that for $V_I(\omega_i,0)$. Hence it is suffices to prove that $V_I^i$ is irreducible i.e. there are no singular vectors in $V_I^i$ except the highest vector. Note that for any $w\in V_I^i$ and $j\in\Delta\backslash I$ we have $E_j(u)w=0$, so any singular vector in $V_I^i$ with respect to $Y_{new}([\fl,\fl])$ is in fact a singular vector of $V(\omega_i,0)$ with respect to $Y_{new}(\fg)$.
\end{proof}

Let us define $Y_V(\fl)$ as the subalgebra of $Y_V(\fg)$ generated by all Fourier components of all matrix elements $t_{\beta,v}(u)$ with $v\in V_I, \beta\in W^\perp = V_I^*$. 
\begin{prop}\label{ylevi}
The image of $Y_V(\fl)$ in $Y_{new}(\fg)$ is $Y_{new}(\fl)$.
\end{prop}
\begin{proof}

Firstly, we prove that 
$Y_V(\fl) \simeq  (Y_{V_I}(\fg_I) \otimes Z) / J$, where $Z$ is the center of $X_{V_I}(\fg_I)$ and $J$ is some ideal.

Each irreducible component $U$ of $V_I \otimes V_I$ is stable under $R= (\rho\otimes\rho)\hat R(-u)$. Moreover, for any $U$ the operator $R$ is a solution of the equation \cite[(3.13)]{wend}
for $U\otimes U$, therefore from \cite[Theorem 3.10]{GRW} it follows that $R$ up to multiplication to $f_U(u) \in \bc[[u^{-1}]]$ is $(\rho_I \otimes \rho_I)\tilde R(-u)$, where $\tilde R(u)$ is the universal $R$-matrix for $Y(\fg_I)$. Hence we have 
$$Y_V(\fl) = X_{V_I}(\fg_I) / J =  (Y_{V_I}(\fg_I) \otimes Z) / J.$$
Next, we claim that the image of $Y_V(\fl)$ contains $Y_{new}(\fl)$.
Indeed, from Proposition \ref{t11} we know that the image contains $e_i^{(1)}$ and $f_i^{(1)}$ for $i \in I$. Note that $e_i^{(1)}, f_i^{(1)}$ together with $H$ generate $Y_{new}(\fl)$. Therefore it is enough to prove that
$\phi^{-1}\circ\psi(Y_V(\fl))$ contains $H$.

Indeed, all $t_{11}^{(r)}$ for different $\omega_i$ commute.
We know that the images of $t_{11}^{(1)}$ and $t_{11}^{(2)}$ under $\phi^{-1}\circ\psi$ belong to $H$.
Then from Proposition \ref{cartan} it follows that in fact $\phi^{-1}\circ\psi(t_{11}^{(r)}) \in H$.
Moreover, the images of $t_{11}^{(r)}$ in $\gr Y_V(\fg)$ are algebraically independent by Proposition \ref{image}. Hence they generate $H$, therefore $H$ lies in the image of $Y_V(\fl)$.
On the other hand, from Proposition \ref{rttgen} it follows that the images of the generators of $Y_{V_I}(\fg_I)$ belong to $Y_{new}(\fl)$, hence the image of $Y_{V_I}(\fg_I)$ lies in $Y_{new}(\fl)$. Moreover, the image of $Z$ belongs to $H$, because $H$ is a maximal commutative subalgebra. Hence $\phi^{-1} \circ \psi (Y_V(\fl)) = Y_{new}(\fl)$.
\end{proof}

\section{Bethe subalgebras}

\subsection{Definition of Bethe subalgebra}
Let $\rho_i:Y(\fg)\to\End V(\omega_i,0)$ be the $i$-th fundamental representation of $Y(\fg)$. Let $$\pi_i: V \to V(\omega_i,0)$$ be the projection. 

Let $T^i(u)=\pi_iT(u)\pi_i$ be the submatrix of $T(u)$-matrix, corresponding to $i$-th fundamental representation.
\begin{defn}
\label{bethe}
Let $C \in \tilde G$. Bethe subalgebra $B(C)\subset Y_V(\fg)$ is the subalgebra generated by all Fourier coefficients of the following series  with the coefficients in $Y_V(\fg)$
$$\tau_i(u,C) = \tr_{V(\omega_i,0)} \rho_i(C)T^i(u),\ \  1 \leqslant i \leqslant n.$$
\end{defn}

\begin{prop}
The coefficients of $\tau_i(u,C)$, $i = 1,\ldots,n$ pairwise commute.
\end{prop}
\begin{proof}
Let $R_{ij}(u) = (\rho_i \otimes \rho_j) \hat R(u)$ where $\hat R(u)$ is the universal $R$-matrix as in Definition \ref{extyangian}, and 
$\rho_i$ is the representation homomorphism corresponding to $V(\omega_i,0)$. Then
$$R_{ij}(u-v)T^i_1(u) T^j_2(v) = T^j_2(v) T^i_1(u) R_{ij}(u-v).$$


Moreover, we know that
$[\rho_i(C)\otimes \rho_j(C), R_{ij}(u - v)] = 0$, because $R_{ij}(u-v)$ is a morphism of $\tilde G$-modules.

Hence we have
$$\rho_i(C)_1T^i_1(u)\rho_j(C)_2 T^j_2(v) = R_{ij}(u-v)^{-1} \rho_i(C)_1 T^j_2(u) \rho_j(C)_2 T^i_1(v) R_{ij}(u-v).$$

Taking the trace over $V(\omega_i,0) \otimes V(\omega_j,0)$ we get
$$\tau_i(u,C)\tau_j(v,C) =  \tau_j(v,C) \tau_i(u,C).$$

\end{proof}

\begin{rem}
Note that the Bethe subalgebra $B(C)$ depends only on the class of $C$ in $\tilde G / Z(\tilde G)$, so in fact Bethe subalgebras are parameterized by $G$.
\end{rem}
\begin{rem}
We define a Bethe subalgebra for the Yangian of a reductive Lie algebra $\fg$ as the tensor product of Bethe subalgebra in $Y_V([\fg,\fg])$ and the center.  
\end{rem}

\begin{rem}
In the work \cite[Section 5]{mo} Maulik and Okounkov define the Yangian $\mathrm{Y}_Q$ of a quiver $Q$ in the following way. We have a collection of vector spaces $F_i(u)$ which are the equivariant cohomology spaces of the corresponding Nakajima quiver variety $\mathcal{M}(w_i)$ where $w_i$ is $1$ at the $i$-th vertex of $Q$ and $0$ at the others. The rational $R$-matrix $R_{ij}(u,v)\in \End(F_i(u)\otimes F_i(v))$ is defined geometrically using the formalism of stable envelopes. The algebra $\mathrm{Y}_Q$ is generated by spaces $\End(F_i(u)\otimes F_i(v))$ modulo quadratic RTT relations determined by $R_{ij}(u,v)\in \End(F_i(u)\otimes F_i(v))$ in any tensor product $F_i(u)\otimes F_j(v)$. For $Q$ of finite type corresponding to a simple Lie algebra $\fg$, the spaces $F_i(u)$ are known to be $V(\omega_i,u)$, the fundamental representations of $Y(\fg)$, so the Yangian $\mathrm{Y}_Q$ is in fact $X_V(\mathfrak{g})$ with $V$ being the sum of fundamental representations. In \cite[Section 6.5]{mo} Maulik and Okounkov define commutative subalgebras in $\mathrm{Y}_Q$, called \emph{Baxter subalgebras}, depending on an element of the maximal torus of the Lie algebra corresponding to $\mathrm{Y}_Q$. This definition coincides with our definition of Bethe subalgebras $B(C)$ with $C\in T$ (more precisely, our $B(C)$ is the image of the corresponding Baxter subalgebra under the quotient homomorphism from $X_V(\fg)$ to $Y_V(\fg)$).
\end{rem}

\subsection{Bethe subalgebras in $\mathcal{O}(G_1[[t^{-1}]])$}

Let $\{V_{\omega_i}\}_{i=1}^{n}$ be the set of all fundamental representations of $\fg$. We also consider  $\{V_{\omega_i}\}_{i=1}^{n}$ as a representations of the corresponding simply-connected group $\tilde G$.

\begin{defn}
Let $C \in \tilde G$. The Bethe subalgebra $\bar B(C)$ of $\mathcal{O}(G_1[[t^{-1}]])$ is the subalgebra generated by of the coefficients of the following series:
$$\sigma_i(u,C) = \tr_{V_{\omega_i}} \rho_i(C)\rho_i(g) = \sum_{v \in \Lambda_i} \Delta_{v,v^{*}}(Cg) = \sum_{r=0}^{\infty} \sum_{v \in \Lambda_i} \Delta_{v,v^{*}}^{(s)}(Cg) u^{-r},$$
where $\Lambda_i$ is some basis of $V_{\omega_i}$, $g \in G_1[[t^{-1}]]$.
\end{defn}
\begin{rem}
Note, that one can define the same subalgebra as the one generated by such traces for \emph{all} finite-dimensional representations of $G$.
\end{rem}
\begin{prop}
For regular $C\in \tilde G$, the Fourier coefficients of $\sigma_i(u,C)$ are algebraically independent. 
\end{prop}
\begin{proof}
Denote by $\sigma_i(C)^{(r)}$ the coefficient of $u^{-r}$ in $\sigma_i(C)$. It is enough to prove that differentials of all $\sigma_i(C)^{(r)}$  at a certain point of $G_1[[t^{-1}]]$ are linearly independent. Note that differential of $\sigma_i(C)^{(r)}$ at the point $e \in G_1[[t^{-1}]]$ is naturally a linear functional on the tangent space $T_e G_1[[t^{-1}]] = t^{-1}g[[t^{-1}]]$.

Let $\chi_{\omega_i}\in\mathcal{O}(\tilde G)$ be the character of the $\tilde G$-module $V_{\omega_i}$. We identify the tangent space $T_{C^{-1}}\tilde G$ with $\fg=T_e\tilde G$ by the left $\tilde G$-action. Then we have the following 

\begin{lem}
\label{gr2}
$d_e \, \sigma_i(C)^{(r)} (xt^{-s}) = \delta_{r,s} d_{C^{-1}} \chi_{\omega_i} (x)$ for any $x\in\fg$. 
\end{lem}
\begin{proof}
We have $$d_e \sigma_i(C)^{(r)}=\sum\limits_{v\in\Lambda_i} d_e \Delta_{v,v^{*}}^{(r)}(C) (xt^{-s}) = \delta_{r,s} \tr_{V_{\omega_i}} \rho_i(C) \rho_i(x)=\delta_{r,s} d_{C^{-1}} \chi_{\omega_i} (x)$$ for any $x\in\fg$. 
\end{proof}

On the other hand the differentials of the characters $\chi_{\omega_i}$ at any regular point $g\in \tilde G$ are algebraically independent, see  e.g. \cite[Theorem 3, p.119]{steinberg} so we are done.
\end{proof}

\begin{prop}
\label{gr}
We have $\gr B(C) = \bar B(C)$ for regular $C\in\tilde G$. In particular, $\bar B(C)$ is Poisson commutative.
\end{prop}
\begin{proof}
From Proposition \ref{image} it follows that $$\gr \tau_i(u,C) = \sigma_i(u,C) + \text{traces for representations with highest weight\ } \mu < \omega_i.$$
Any irreducible representation can be realized as the highest subrepresentation in some tensor product of the fundamental representations. Therefore the subalgebra generated by all $\gr \tau_i(u,C)$ coincides with the subalgebra generated by $\sigma_i(u,C)$.
Moreover, by Proposition~\ref{gr2} the coefficients of $\sigma_i(u,C)$ are algebraically independent, hence $\gr B(C) = \tilde B(C)$.
\end{proof}

\begin{thm}
\label{thm1}
For any regular $C \in \tilde G$, the subalgebra $B(C) \subset Y(\fg)$ is a free polynomial algebra. Moreover, the Poincar\'e series of $B(C)$ coincides with the Poincar\'e series of $H$.
\end{thm}
\begin{proof}
This follows from Proposition~\ref{gr}. 
\end{proof}

\subsection{Limits of Bethe subalgebras}
Following \cite{ir} it is possible to construct new commutative subalgebras as limits of Bethe subalgebras.
We describe here what the "limit" means.
Let $C$ be an element of $G^{reg}$. Consider $B^{(r)}(C):= Y_V^{(r)} \cap B(C)$. We have proved that the images of the coefficients of  $\tau_1(u,C), \ldots,$ $\tau_n(u,C)$ freely generate the subalgebra $\bar B(C) = \gr B(C)\subset \gr Y_V(\fg)$. Hence the dimension $d(r)$ of $B^{(r)}(C)$ does not depend on $C$.
Therefore for any $r \geqslant 1$ we have a map $\theta_r$ from $G^{reg}$ to $\prod_{i=1}^r {\rm Gr}(d(i),\dim Y_V^{(i)})$ such that $C \mapsto (B^{(1)}(C), \ldots, B^{(r)}(C))$. Here ${\rm Gr}(d(i),\dim Y_V^{(i)})$ is the Grassmanian of subspaces of dimension $d(i)$ of a vector space of dimension  $\dim Y_V^{(i)}$.
Denote the closure of $\theta_r(G^{reg})$ (with respect to Zariski topology) by $Z_r$. There are well-defined projections  $\zeta_r: Z_r \to Z_{r-1}$ for all $r \geqslant 1$. The inverse limit $Z = \varprojlim Z_r$ is well-defined as a pro-algebraic scheme and is naturally a parameter space for some family of commutative subalgebras which extends the family of Bethe subalgebras.

Indeed, any point $z\in Z$ is a sequence $\{z_r\}_{r \in \mathbb{N}}$ where $z_r \in Z_r$ such that $\zeta_r(z_r) = z_{r-1}$. Every $z_r$ is a point in $\prod_{i=1}^r {\rm Gr}(d(i),\dim Y_V^{(i)})$ i.e. a collection of subspaces $B_{r}^{(i)}(z)\subset Y_V^{(i)}$ such that $B_{r}^{(i)}(z)\subset B_{r}^{(i+1)}(z)$ for all $i<r$. Since $\zeta_r(z_r) = z_{r-1}$ we have $B_{r}^{(i)}(z)= B_{r-1}^{(i-1)}(z)$ for all $i<r$. Let us define the subalgebra corresponding to $z\in Z$ as $B(z):= \bigcup_{r=1}^{\infty} B_{r}^{(r)}(z)$. 

\begin{prop}\label{pr-limit} $B(z)$ is a commutative subalgebra in $Y_V(\fg)$ with the same Poincar\'e series as $B(C)$ for $C\in G^{reg}$. We call it a limit subalgebra.
\end{prop}

\begin{proof} Indeed, $B(z)$ is a commutative subalgebra because being a commutative subalgebra is a Zariski-closed condition: we have $B_r(C)\cdot B_s(C)\subset B_{r+s}(C)$ for all $C\in G^{reg}$ for all $r,s$ and this product is commutative, hence we get the same for the product $B_{r}^{(r)}(z)\cdot B_{s}^{(s)}(z)=B_{r+s}^{(r)}(z)\cdot B_{r+s}^{(s)}(z)\subset B_{r+s}^{(r+s)}(z)$.
\end{proof}

\section{Wonderful compactification and Bethe subalgebras.}
\subsection{Wonderful compactification}
Following \cite{evens} we recall the construction of $\oG$, De-Concini - Procesi wonderful compactification of the Lie group $G$.
Let $V$ be a representation of $\tilde G$. This defines the map 
$$G \to \bp(\End V).$$ If $V = V_{\lambda}\oplus \bigoplus_{\mu < \lambda} V_{\mu}$ and $V_{\lambda}$ is irreducible of regular highest weight $\lambda$, then the closure of the image of $G$ in $\bp(\End V)$ is a smooth projective variety called De-Concini - Procesi wonderful compactification $\oG$. Note that there is an action of $G \times G$ on $\oG$ by 
$$(g_1,g_2) \cdot x = g_1 x g_2^{-1}.$$
There is the following theorem about the orbit structure on $\oG$. To any $I \subset \Delta$ one can assign the pair of opposite parabolic subgroups $P_I, P_I^-$ and the Levi subgroup $L_I=P_I\cap P_I^-$. Denote by $G_I = L_I / Z(L_I)$. Denote by $V_I$ the $L_I$-subrepresentation in $V$ generated by the highest vector of $V$.
There is a natural embedding of $\bp(\End(V_I)) \to \bp(\End(V))$. The subgroup $G_I$ is embedded into $\bp(\End(V))$ by the map 
$$G_I \to \bp(\End(V)),\ g \mapsto [g \circ{pr_I}], $$
where $pr_I: V \to V_I$ is the $\tilde G$-invariant projection.
\begin{thm} (see \cite{evens})
\begin{enumerate}
    \item The orbits of $G \times G$ on $\oG$ are in bijection with subsets of $\Delta$. Let $S_I^{0}$ be the orbit corresponding to $I \subset \Delta$ and $S_I$ be its closure. Then $S_I^0 \subset S_J$ if and only if $I \subset J$.
    \item   
    The maps $(G \times G) \times_{P_I \times P_I^-} G_I \to S_I^0, (g_1,g_2,x) \mapsto g_1xg_2^{-1}$ and $(G \times G) \times_{P_I \times P_I^-} \oG_I \to S_I, (g_1,g_2,x) \mapsto g_1xg_2^{-1} $ are isomorphisms of $G\times G$-varieties.
\end{enumerate}
\end{thm}

Note that the stratum $S_I^0$ is a homogeneous bundle over $G/P_I \times G/P_I^{-}$ with the fiber $G_I$. Let $$\varphi: S_I^0 \to G/P_I \times G / P_I^{-}$$
be the projection. We identify $G/ P_I$ with the set of parabolic subalgebras $\fp\subset \fg$ conjugate to $\fp_I$. Similarly, $G/ P_I^-$ is the set of parabolic subalgebras $\fp^-\subset \fg$ conjugate to $\fp_I^-$. Note that $G/P_I \times G / P_I^{-}$ contains a unique open $G$-orbit. 

A pair of parabolic subalgebras $(\fp,\fp^-)$ belongs to the open $G$-orbit if and only if $\fp+\fp^-=\fg$. Indeed, the $G$-orbit of $(\fp,\fp^-)$ is open if and only if the projection map 
$$\fg \to \fg \oplus \fg \to \fg/\fp_I \oplus \fg/\fp_I^-,$$
is surjective, which is equivalent to the condition $\fp+\fp^-=\fg$.

Consider the preimage of this open $G$-orbit on $G/P_I \times G / P_I^{-}$ in the stratum $S_I^0$:
$$S_I^{00} = \bigcup_{\exists g\in G: \Ad(g) \fp = \fp_I, \Ad(g) \fp^{-} = \fp_I^-} \varphi^{-1}(\fp,\fp^-)$$ 
Clearly $S_I^{00}$ is a Zariski open subset of $S_I^0$. 

Let $T$ be a maximal torus of $G$ and $\overline{T}$ be its closure in the wonderful compactification.
\begin{prop}
\label{torus}
$\overline{T}$ belongs to $\bigcup_I S_I^{00}$.
\end{prop}
\begin{proof}
Suppose that $x \in \overline{T}$. Here we consider $x$ as a point in $\bp(\End V)$.
Note that in $\oG$ different points have different stabilizers in $G\times G$. Therefore, if we find the stabilizer of $x$ then we find the corresponding point in $\bigcup_I S_I$. 
Consider $R = \{ g \, | \, gx=x \}$ and $R^- = \{ g \, | \, xg^{-1}=x \}$.
Then $R$ and $R^{-}$ are radicals of the parabolic subgroups $P$ and $P^{-}$, corresponding to the point $x$. 
Next, we recover the parabolics $P, P^-$ as the normalizers of $R, R^-$, respectively.
We claim that corresponding subalgebras $\fp$ and $\fp^-$ has the property, that $\fp + \fp^- = \fg$, therefore this point belongs to $\bigcup_I S_I^{00}$.
To prove it we need the following antiautomorphism which is constant on $\fh$
$$\sigma: \fg \to \fg, \quad \fg_{\alpha} \to \fg_{-\alpha} \ \text{for} \ \alpha \in \Phi$$
and a non-degenrate Shapovalov form on $V$ with the property
$$(xv,w) = (v, \sigma(x)w)$$
for any $x \in \fg$ and $v,w \in V$.

With respect to this form, all operators from the image of $T$ are self-adjoint, hence the same is true for all operators from $\overline{T}$, so we have $\sigma(R) = R^-$.
But then $\fp = \sigma(\fp^-)$ and hence $\fp+\fp^- = \fg$, because $\fh \subset \fp$.

\end{proof}

\subsection{Bethe subalgebras}
We assign a commutative subalgebra of the Yangian to any $X \in \overline{G}$. Let $$V = \bigoplus_i V(\omega_i, 0)$$
be again the sum of fundamental representations of the $Y(\fg)$.
If we consider $V$ as a representation of Lie algebra $\fg$, then $$V = \left(V_{\omega_1} \oplus \left(\bigoplus_{\mu_1 < \omega_1} V_{\mu_1}\right)\right) \oplus \ldots \oplus \left(V_{\omega_n} \oplus \left(\bigoplus_{\mu_n < \omega_n} V_{\mu_n}\right)\right).$$
Consider the Segre embedding 
$$\theta: G \to \prod_i \bp(\End(V(\omega_i,0))) \to \bp\left(\End\left(\bigotimes_i V(\omega_i,0)\right)\right).$$
Note that $$\End\left(\bigotimes_i V(\omega_i,0)\right) =  \End(V_{\omega_1 + \ldots + \omega_n}) \oplus \End(\oplus_{\mu < \sum_i \omega_i} V_{\mu}).$$
The closure of the image of $G$ under $\theta$ is the same as that of $G$ in $\prod_i \bp(\End(V(\omega_i,0)))$. So the latter is isomorphic to the wonderful closure $\oG$, see \cite[Proposition 3.5]{evens}.

\subsection{Limit subalgebras and the wonderful compactification}

Suppose that $X = (X_1, \ldots, X_n) \in \oG\subset \prod_i \bp(\End(V(\omega_i,0)))$.
We define the subalgebra $B(X)$ of $Y_V(\fg)$ using the same formulas as in Definition \ref{bethe}, just changing $\rho_i(C)$ to $\hat X_i$:
$$\tau_i(u,X) = \tr_{V(\omega_i,0)} X_i T^i(u),\ \  1 \leqslant i \leqslant n.$$
Here $\hat X_i$ is an arbitary representative of $X_i$ in $\End(V(\omega_i,0))$. Note that this definition is correct because $\tau_i(u,X) = \tau_i (u, cX)$ for any $c \in \bc / \{0\}$.
Note that Poincar\'e series of $B(X)$ could be smaller that that of $B(C)$ with $C \in G^{reg}$.

On the other hand we have the scheme $Z$ parameterizing all limit Bethe subalgebras. Let $\tilde{Z}$ be the closure of $G^{reg}$ in $Z\times\oG$.

\begin{prop}
For any $(z,X) \in \tilde{Z}\subset Z\times\oG$ we have $B(X)\subseteq B(z)$. In particular subalgebra $B(X)\subset Y(\fg)$ is commutative.
\end{prop}
\begin{proof}
$B(X)\subseteq B(z)$ is a Zariski closed condition on $(z,X)\in \tilde{Z}\subset Z\times\oG$ which is satisfied on the open dense subset $G^{reg}\subset\tilde{Z}$.
\end{proof}

Our purpose is to describe these subalgebras explicitly for any $X\in \bigcup_I S_I^{00}$.
Consider a point $X=(g, g, x) \in S_I^{00}$, where $g \in G, x = \tilde x \circ \pr_I \in \bp(\End(V_I)), \tilde x \in G_I$. Let $\fl = \Ad(g)\fp_I \cap \Ad(g)\fp_I^{-}$ be the corresponding Levi subalgebra.
\begin{thm}
\label{bethelevi}
For any $X=(g, g, x) \in S_I^{00}$, the corresponding subalgebra $B(X)$ is $B(g\hat x g^{-1}) \subset Y_V(\fl) \subset Y_V(\fg)$, where $\hat x \in L_I$ such that the class $[\hat x]$ in $L_I / Z(L_I)$ is $\tilde x$. In particular, for any $(z,X)\in \tilde{Z}$ such that $X\in S^{00}$ and $x\in G_I^{reg}$ we have $B(z)=B(X)$.
\end{thm}
\begin{proof} We have
\begin{gather*}
\tau_i(u,X) = \tr_{V(\omega_i,0)} \rho_i(g)x\rho_i(g^{-1})T^i(u) = \\ \tr_{V(\omega_i,0)} \rho_i(g) x \Ad(g^{-1})T^i(u)\rho_i(g^{-1}) = \rho_i(g)\tr_{V(\omega_i,0)}x \Ad(g^{-1})T^i(u) \rho_i(g^{-1}) = \\ \rho_i(g)\tr_{V_I \cap V(\omega_i,0)} \rho_I(\hat x) \Ad(g^{-1})T^i(u) \rho_i(g^{-1})  = \tr_{V_I \cap V(\omega_i,0)} \rho_i(g)\rho_I(\hat x)\rho_i(g^{-1})T^i(u).
\end{gather*}
so the generators of $B(X)$ belong to $B(gxg^{-1}) \subset Y_V(\fl)$.
Note that $\tau_i(u,X)$ with $i \in Q \setminus I$ generate the center of $Y_V(\fl)$ and $\tau_i(u,X)$ with $i \in I$ generate Bethe subalgebra in $Y_V([\fl,\fl])$. Since $B(X)$ is conjugate to the Bethe subalgebra $B(x)$ in $Y_V(\fl)$ it has the same Poincar\'e series as the subalgebra $H$. Hence the Poincar\'e series of $B(X)$ and $B(z)$ are the same, so $B(X)=B(z)$. 
\end{proof}

\begin{rem}
By Proposition~\ref{torus} the above assertion is true for any $X\in \overline{T}$.
\end{rem}

\begin{cor}
For $X=(g,g,x)\in S^{00}_I$ such that $x\in G_I$ is regular, the corresponding Bethe subalgebra $B(X)$ is a free polynomial algebra with the same Poincar\'e series as that of the Cartan subalgebra $H\subset Y(\fg)$.
\end{cor}

We conjecture that the scheme $Z$ which parameterizes limit Bethe subalgebras is a resolution of $\oG$. More precisely, we have the following 

\begin{conj}
The projection to the first factor of $Z\times \oG$ gives the isomorphism $\tilde{Z}\tilde{\to}Z$. The projection to the second factor gives a birational proper map $\tilde Z\to \oG$.
\end{conj}

\section{Example $\fg= \fsl_n$}\label{se-example}

\subsection{Definition of Bethe subalgebras}
Suppose that $\fg = \fsl_n$. We consider $W = V(\omega_1,0)$. As a representation of $\fg$ this is the tautological representation $\bc^n$.
In this case $X_W(\fg)$ is well-known $Y(\fgl_n)$, see \cite{molev}. Note that $R(u) = 1 - \frac{P}{u}$ in this representation. 
Here $P$ is the permutation operator in $\End W \otimes \End W$.

Let $C$ be a diagonal matrix with the eigenvalues $\lambda_1,\ldots,\lambda_n$.
There are two equivalent definitions of a Bethe subalgebra. 
\begin{defn}
Bethe subalgebra $B(C) \subset Y_{W}(\fsl_n)$ is the subalgebra generated by coefficietns of the following series, $k=1, \ldots, n$
$$\tau_k(u,C) = \sum_{1 \leqslant a_1< \ldots < a_k \leqslant n} \lambda_{a_1} \ldots \lambda_{a_k} t_{a_1 \ldots a_k}^{a_1\ldots a_k}(u),$$
where $$t^{a_1 \ldots a_k}_{b_1 \ldots b_k}(u) = \sum_{\sigma \in S_k} (-1)^\sigma \, \cdot \, t_{a_{\sigma(1)}b_1}(u) \ldots t_{a_{\sigma(k)}b_k}(u-k+1)$$ are the quantum minors.
\end{defn}
Denote by $A_k$ the image of antisymmetrizer $$\sum_{\sigma \in S_k} (-1)^{\sigma}\sigma \in \mathbb{C}[S_k]$$ in $\End(\bc^n)^{\otimes k}$.
\begin{defn}\label{bethe-sln}
Bethe subalgebra $B(C) \subset Y_{W}(\fsl_n)$ is the subalgebra generated by coefficietns of the following series, $k=1, \ldots, n$
$$\tau_k(u,C) = \tr A_k  C_1 \ldots C_k T_1(u) \ldots T_k(u-k+1),$$
where we take the trace over all $k$ copies of $ \rm{End} \ \mathbb{C}^n$.
\end{defn}
It is well-known that in this case $V(\omega_i,0)$ as $\fsl_n$-module is $V_{\omega_i}=\Lambda^i(\bc^n)$.
Therefore $V = \bigoplus_i V_{\omega_i}$. 
Moreover, the $Y(\fsl_n)$-module $V(\omega_k,0)$ is the image of the operator $A_k$ in the tensor product of $Y(\fsl_n)$-modules $$ V(\omega_1, -(k-1)) \otimes V(\omega_1, -(k-2)) \otimes \ldots \otimes V(\omega_1, 0) = (\bc^n)^{\otimes k},$$ see \cite[Proposition 6.5.1]{molev}.
\begin{rem}
Note that our action of $T(u) \in \End W \otimes X_W(\fg)$ in representation $W$ is $R(u)$. It is different from the action in \cite{molev}, where $T^{t}(u)$ acts by $R(-u)$.
\end{rem}
Recall that $T^i(u) \in \End V \otimes Y_V(\fg)$ is the submatrix of $T(u)$-matrix, corresponding to $i$-th fundamental representation.
It is easy to see that $T^1(u)$ maps to $T(u)$ under the isomorphism $Y_V(\fsl_n) \simeq Y_{W}(\fsl_n)$. 
Therefore from the inclusion of representation it follows that $T^k(u) $ maps to $T_k(u-k+1) \ldots T_1(u) A_k = A_k T_1(u) \ldots T_k(u-k+1)$.  


So we see that our Definition~\ref{bethe} of Bethe subalgebras coincides with the Definition~\ref{bethe-sln}.

\subsection{Limits of Bethe subalgebras}
Let us define two different embeddings:
$$\iota_k: Y(\fgl_{n-k}) \to Y(\fgl_{n}); \ t^{(r)}_{ij} \mapsto t^{(r)}_{ij}$$
$$\xi_{n-k}: Y(\fgl_k) \to Y(\fgl_{n}); \ t^{(r)}_{ij} \mapsto t^{(r)}_{n-k+i,n-k+j}$$
According to PBW-theorem these maps are injective.
By definition, put
$$\gamma_n: Y(\fgl_n) \to Y(\fgl_n); \ T(u) \mapsto (T(-u-n))^{-1}.$$
It is well-known that $\gamma_n$ is an involutive automorphism of $Y(\fgl_n)$.
We define a homomorphism
$$\eta_{n-k} = \gamma_{n} \circ \xi_{n-k} \circ \gamma_k: Y(\fgl_k) \to Y(\fgl_{n}).$$
Note that $\eta_{n-k}$ is injective.

Recall the description of some limit subalgebras of $Y(\fgl_n)$ from \cite{ir}, e.g. \cite[Theorem 5.3(1)]{ir}. 
\begin{thm}
Let $C_0=\diag(\lambda_1, \ldots, \lambda_{n-k})$ and $C_1=\diag(\lambda_{n-k+1}, \ldots, \lambda_n)$.  Suppose that
$C(t) = \diag(C_0, \underbrace {0, \ldots, 0}_k)
+ t\cdot\diag(\underbrace{0, \ldots, 0}_{n-k}, C_1) \in T^{reg}$ (i.e. both $C_0$ and $C_1$ are regular and non-degenerate). Then
$$\lim_{t \to 0} B(C(t)) = \iota_k(B(C_0)) \otimes \eta_{n-k}(B(C_1)).$$
\end{thm}

We give a new proof of this Theorem by deducing it from Theorem~\ref{bethelevi}.
Note that $Y(\fgl_n) = X_W(\fsl_n) = Y_W(\fsl_n) \otimes Z(X_W(\fsl_n))$, see e.g. \cite{molev}. Moreover $B(C) \subset Y(\fgl_n)$ contains the center and is a tensor product of a Bethe subalgebra in the $Y_{W}(\fsl_n)$ and the center. Hence it is sufficient to prove Theorem 6.5 in case of $Y_W(\fsl_n)$. 

\begin{proof}
Let $I = \Delta / \{\alpha_k\}$. Let $\fl$ be the corresponding Levi subalgebra. Consider the limit $t \to 0$ of $C(t)$ in $\overline{T}\subset\oG$. Then we obtain the point $X = (e,e, x) \in S_I^{00}$, where $$x= \diag\left(\lambda_1, \ldots, \lambda_{n-k}, \left(\prod_{i=1}^{n-k} \lambda_i\right) \lambda_{n-k+1}, \ldots, \left(\prod_{i=1}^{n-k} \lambda_i\right) \lambda_n\right).$$  

We claim that $B(X)$ coincides with $\iota_k(B(C_0)) \otimes \eta_{n-k}(B(C_1))$. Indeed, from section \ref{levisub} we have an embedding of $Y_{V_I}(\fl)  \to Y_V(\fsl_n) \simeq Y_W(\fsl_n)$ such that the image of $Y_{V_I}(\fl)$ in $Y_W(\fsl_n)$ is generated by the following elements:
$$t_{ij}(u), 1 \leq i,j \leq n-k \quad \text{from} \ V_{\omega_1} $$
$$t_{i_1,i_2}^{i_1,i_2}(u), 1 \leq i_1,i_2 \leq n-k \quad \text{from} \ V_{\omega_2}$$
\begin{center}\ldots\end{center}
$$t_{1,\ldots,n-k}^{1,\ldots,n-k}(u) \quad \text{from} \ V_{\omega_k}$$
$$t_{1, \ldots, n-k, n-k+j}^{1, \ldots, n-k, n-k+i}(u), 1 \leq i,j \leq k \quad  \text{from} \ V_{\omega_{k+1}}$$
\begin{center}\ldots\end{center}
$$t_{1,\ldots,n}^{1, \ldots, n}(u) = 1 \quad \text{from} \ V_{\omega_n}$$

Therefore Bethe subalgebra corresponding to $X$ is generated by the coefficients of the following series:
$$\tau_1(u,X) = \sum_{1 \leqslant a_1 \leqslant n-k} \lambda_{a_1} t_{a_1 a_1}(u)$$
$$\tau_2(u,X) = \sum_{1 \leqslant a_1< a_2 \leq n-k} \lambda_{a_1} \lambda_{a_2} t_{a_1 a_2}^{a_1 a_2}(u)$$
\begin{center}\ldots\end{center}
$$\tau_{n-k}(u,X) = \lambda_{1} \ldots \lambda_{n-k} t_{1 \ldots n-k}^{1 \ldots n-k}(u)$$
$$\tau_{n-k+1}(u,X) =  \left(\prod_{i=1}^{n-k} \lambda_i\right) \sum_{1 \leqslant a_1 \leqslant k}  \lambda_{n-k+a_1} t_{1, \ldots, n-k, n-k+a_1}^{1, \ldots, n-k, n-k+a_1}(u)$$
$$\tau_{n-k+2}(u,X) = \left(\prod_{i=1}^{n-k} \lambda_i\right) \sum_{1 \leqslant a_1 < a_2 \leqslant k}  \lambda_{n-k+a_1} \lambda_{n-k+a_2} t_{1, \ldots
, n-k, n-k+a_1, n-k+a_2}^{1, \ldots, n-k, n-k+a_1, n-k+a_2}(u)$$
\begin{center}\ldots\end{center}
$$\tau_{n}(u,X) = \lambda_{1} \ldots \lambda_n t_{1 \ldots n}^{1 \ldots n}(u) = \lambda_1 \ldots \lambda_n.$$

The coefficients of the first $n-k$ series generate subalgebra $\iota_k(B(C_0))$.
Next we have $t_{1 \ldots n-k}^{1 \ldots n-k}(u)$ in our subalgebra $B(X)$ hence $t_{1 \ldots n-k}^{1 \ldots n-k}(u)^{-1}$ belongs to $B(X)$ as well. 

Recall the following Lemma, see \cite[Lemma 1.11.3]{molev}.
\begin{lem}
Consider the map $\eta_{n-k}: Y(\fgl_k) \to Y(\fgl_n)$. We have 
$$\eta_{n-k}(t_{b_1, \ldots, b_m}^{a_1,\ldots,a_m}(u)) = (t_{1 \ldots n-k}^{1 \ldots n-k}(u+n-k))^{-1} \cdot t_{1, \ldots, n-k, n-k+b_1, \ldots, n-k+b_m}^{1, \ldots, n-k, n-k+a_1, \ldots, n-k+a_m}(u+n-k)$$
\end{lem}
From this lemma we see that generators of Bethe subalgebra $\eta_{n-k}(B(C_1)) \subset Y_W(\fsl_n)$ are $$t_{1 \ldots n-k}^{1 \ldots n-k}(u+n-k)^{-1} \cdot \tau_{n-k+1}(u+n-k,X), \ldots, t_{1 \ldots n-k}^{1 \ldots n-k}(u+n-k)^{-1} \cdot \tau_n(u+n-k,X).$$ 
We conclude that $B(X) = \lim_{t \to 0} B(C(t))$.

\end{proof}






\newpage
\footnotesize{
{\bf Aleksei Ilin} \\ National Research University
Higher School of Economics, \\
Russian Federation,\\
Department of Mathematics, 6 Usacheva st, Moscow 119048;\\
{\tt alex\_omsk@211.ru}}
\\
\footnotesize{
{\bf Leonid Rybnikov} \\
National Research University
Higher School of Economics,\\ Russian Federation,\\
Department of Mathematics, 6 Usacheva st, Moscow 119048;\\
Institute for Information Transmission Problems of RAS;\\
{\tt leo.rybnikov@gmail.com}} \\

\end{document}